\documentclass[11pt]{amsart}
\usepackage{amssymb, enumitem}
\usepackage{hyperref, aliascnt}
\usepackage[english]{babel}
\usepackage{enumitem}
\usepackage{comment}
\usepackage{tikz}
\usetikzlibrary{shapes,decorations.pathreplacing}
\usetikzlibrary{calc}
\usetikzlibrary{backgrounds}
\usetikzlibrary{automata}
\usetikzlibrary{positioning}
\usetikzlibrary{decorations.markings}
\usetikzlibrary{arrows}
\usetikzlibrary{scopes}
\usetikzlibrary{intersections}
\usetikzlibrary{fit}

\def\today{\number\day\space\ifcase\month\or   January\or February\or
   March\or April\or May\or June\or   July\or August\or September\or
   October\or November\or December\fi\   \number\year}

\theoremstyle{definition}
\newtheorem{lma}{Lemma}[section]

\newaliascnt{thmCt}{lma}
\newtheorem{thm}[thmCt]{Theorem}
\aliascntresetthe{thmCt}

\newaliascnt{corCt}{lma}
\newtheorem{cor}[corCt]{Corollary}
\aliascntresetthe{corCt}

\newaliascnt{propCt}{lma}
\newtheorem{prop}[propCt]{Proposition}
\aliascntresetthe{propCt}

\newtheorem*{thm*}{Theorem}
\newtheorem*{cor*}{Corollary}
\newtheorem*{prop*}{Proposition}

\newcounter{theoremintro}
\newtheorem{thmintro}[theoremintro]{Theorem}
\newtheorem{corintro}[theoremintro]{Corollary}

\newaliascnt{pgrCt}{lma}

\aliascntresetthe{pgrCt}

\newaliascnt{dfCt}{lma}
\newtheorem{df}[dfCt]{Definition}
\aliascntresetthe{dfCt}

\newaliascnt{remCt}{lma}
\newtheorem{rem}[remCt]{Remark}
\aliascntresetthe{remCt}

\newaliascnt{remsCt}{lma}

\aliascntresetthe{remsCt}

\newaliascnt{egCt}{lma}
\newtheorem{eg}[egCt]{Example}
\aliascntresetthe{egCt}

\newaliascnt{egsCt}{lma}

\aliascntresetthe{egsCt}

\newaliascnt{qstCt}{lma}

\aliascntresetthe{qstCt}

\newaliascnt{pbmCt}{lma}

\aliascntresetthe{pbmCt}

\newaliascnt{notaCt}{lma}

\aliascntresetthe{notaCt}

\newcommand{\beq}{\begin{equation}}
\newcommand{\eeq}{\end{equation}}
\newcommand{\beqa}{\begin{eqnarray*}}
\newcommand{\eeqa}{\end{eqnarray*}}
\newcommand{\bal}{\begin{align*}}
\newcommand{\eal}{\end{align*}}
\newcommand{\bi}{\begin{itemize}}
\newcommand{\ei}{\end{itemize}}
\newcommand{\be}{\begin{enumerate}}
\newcommand{\ee}{\end{enumerate}}

\newcommand{\ep}{\varepsilon}

\newcommand{\Z}{{\mathbb{Z}}}
\newcommand{\C}{{\mathbb{C}}}
\newcommand{\N}{{\mathbb{N}}}

\newcommand{\T}{{\mathbb{T}}}

\newcommand{\X}{{\mathsf{X}}}
\newcommand{\Oo}{{\mathcal{O}}}

\newcommand{\dist}{d}

\newcommand{\Aut}{{\mathrm{Aut}}}


\title{Simplicity of $C^*$-algebras of contracting self-similar groups}

\date{January 20, 2025}

\thanks{
The first named author was supported by a starting grant of
the Swedish Research Council. The second named author was supported by the NSF grant DMS2204379.  The third named author was supported by a Simons Foundation Collaboration Grant, award number 849561, and the Australian Research Council Grant DP230103184.
}

\author[Gardella]{Eusebio Gardella}
\address[E.~Gardella]
{Department of Mathematical Sciences, Chalmers University of
Technology and University of Gothenburg, Gothenburg SE-412 96, Sweden.}
\email{gardella@chalmers.se}
\urladdr{www.math.chalmers.se/~gardella}

\author[Nekrashevych]{Volodymyr Nekrashevych}
\address[V.~Nekrashevych]
{Department of Mathematics\\
Texas A\& M University\\
College Station, TX 77843-3368}
\email{nekrash@tamu.edu}
\urladdr{people.tamu.edu/~nekrash/}

\author[Steinberg]{Benjamin Steinberg}
\address[B.~Steinberg]{%
    Department of Mathematics\\
    City College of New York\\
    Convent Avenue at 138th Street\\
    New York, New York 10031\\
    USA}
\email{bsteinberg@ccny.cuny.edu}

\author[Vdovina]{Alina Vdovina}
\address[A.~Vdovina]{%
    Department of Mathematics\\
    City College of New York\\
    Convent Avenue at 138th Street\\
    New York, New York 10031\\
    USA}
\email{avdovina@ccny.cuny.edu}

\begin{document}

\begin{abstract}
We show that the $C^*$-algebra associated by Nekrashevych to a contracting self-similar group is simple if and only if the corresponding complex $\ast$-algebra is simple.  We also improve on Steinberg and Szaka\'c's algorithm to determine if the $\ast$-algebra is simple.   This provides an interesting class of non-Hausdorff amenable, effective and minimal ample groupoids for which simplicity of the $C^*$-algebra and the complex $\ast$-algebra are equivalent.
\end{abstract}

\maketitle

\section{Introduction}
It has been known for quite some time that if $\mathcal G$ is a Hausdorff amenable ample groupoid, then $C^*(\mathcal G)$ is simple if and only if $\mathcal G$ is minimal and effective, if and only if the algebra $K\mathcal G$ is simple over any (some) field $K$; see~\cite{BCFS14,groupoidprimitive,CEM15}.  If $\mathcal G$ is non-Hausdorff, then minimality and effectiveness are still necessary conditions, but things begin to break apart.  In particular, simplicity of $K\mathcal G$ can depend on the characteristic of $K$~\cite{nonhausdorffsimple,simplicity}, and it is possible for $\C \mathcal G$ and $C^*(\mathcal G)$ to not be simple.  It was observed in~\cite{Nekrashevychgpd} in the algebraic setting and in~\cite{nonhausdorffsimple} for both the algebraic and $C^*$-setting that the obstruction to simplicity is controlled by an ideal that is nowadays called the essential ideal~\cite{KM21,simplicity} (see also~\cite[Section~4]{Nek:jot}).  The algebraic essential ideal of the complex algebra of an ample groupoid is always contained in the essential ideal of the reduced $C^*$-algebra of the groupoid.  However, it is an open question whether the algebraic essential ideal is dense in the essential ideal, and it is also open whether $C^*_r(\mathcal G)$ is simple if and only if $\C\mathcal G$ is simple for a minimal effective ample groupoid.

The second author~\cite{Nek:jot,Nekcstar} introduced a $C^*$-algebra associated to a self-similar group, which moreover has an ample groupoid model which is minimal and effective for any self-similar group. In the case that the self-similar group is contracting, the second author showed that the groupoid is amenable and second countable. On the other hand, the groupoid associated to a self-similar group is seldom Hausdorff.  Thus these groupoids form an interesting test case for understanding simplicity phenomena for non-Hausdorff groupoids.

The most famous contracting self-similar group is Grigorchuk's  finitely generated torsion group of intermediate growth~\cite{GrigPeriodic,GrigGrowth}. The second author showed in~\cite{Nekrashevychgpd} that the algebra of the associated ample groupoid is not simple over fields of characteristic $2$, but in~\cite{nonhausdorffsimple} it was shown that the algebra of this groupoid is simple over fields of all other characteristics and the $C^*$-algebra of the groupoid is simple.  In~\cite{SS23}, the third author and Szaka\'cs provided an algorithm which on input the nucleus of a contracting self-similar group, outputs all characteristics of fields for which the algebra is not simple.  They introduced a class of contracting self-similar groups, called multispinal groups, that generalizes the Grigorchuk group and \v{S}uni\'c groups~\cite{sunicgroups}.  For this class they give an easy to verify criterion for simplicity of its algebra over any field.  Yoshida~\cite{2021arXiv210202199Y} gave a criterion for simplicity of the $C^*$-algebra of a multispinal group that coincides with the criterion in~\cite{SS23} for simplicity of the complex $*$-algebra.

In this paper, we prove that the $C^*$-algebra associated to a contracting self-similar group is simple if and only if the complex $*$-algebra is simple.  We also improve on the efficiency of the algorithm in~\cite{SS23} by showing that if the algebraic essential ideal is nonzero, then there is an element in the ideal supported on a particular type of finite subgroup of the nucleus.  This allows us to obtain an exponential upper bound on the complexity of the algorithm.  The algorithm  in~\cite{SS23} is bounded in terms of Bell numbers.

\subsection*{Acknowledgments}
This work was initiated while the authors were participating in the AIM workshop ``Groups of dynamical origin'', in June 2024. We thank the organizers of the workshop and AIM for the stimulating environment.

\section{Ample groupoids and associated algebras}

Let us review classical notations of groupoid theory.
A groupoid $\mathcal{G}$ is a small category of isomorphisms. We identify objects of the category with their identity automorphisms, and call them \emph{units}. For an element $g\in\mathcal{G}$, the \emph{source} and \emph{range} of $g$ are then $\mathsf{s}(g)=g^{-1}g$ and $\mathsf{r}(g)=gg^{-1}$, and a product $g_1g_2$ of elements of $\mathcal{G}$ is defined if and only if $\mathsf{r}(g_2)=\mathsf{s}(g_1)$.
We denote the set of units of $\mathcal{G}$ by $\mathcal{G}^{(0)}$.

We also use the notation $\mathcal{G}_x=\mathsf{s}^{-1}(x)$, $\mathcal{G}^x=\mathsf{r}^{-1}(x)$ and $\mathcal{G}_x^y=\mathcal{G}_x\cap\mathcal{G}^y$ for units $x, y\in\mathcal{G}^{(0)}$.

A \emph{topological groupoid} is a groupoid with a topology on it such that the multiplication and taking inverse are continuous maps.

A \emph{bisection} is a subset $A\subset\mathcal{G}$ such that the maps $\mathsf{s}\colon A\to\mathsf{s}(A)$ and $\mathsf{r}\colon A\to\mathsf{r}(A)$ are homeomorphisms. A groupoid is said to be \'etale if $\mathsf{s}, \mathsf{r}$ are open and it has a basis for its topology consisting of open bisections. (Equivalently, if both maps $\mathsf{s}, \mathsf{r}\colon \mathcal{G}\to\mathcal{G}^{(0)}$ are local homeomorphisms.)

If $\mathcal{G}$ is \'etale and $C\subset\mathcal{G}$ is compact, then for every unit $x$ the sets $C\cap\mathcal{G}_x$ and $C\cap\mathcal{G}^x$ are finite, since $C$ can be covered by finitely many bisections.

In this paper, all \'etale groupoids have locally compact Hausdorff unit spaces, but need not themselves be Hausdorff.

A groupoid $\mathcal G$ is said to be
\emph{ample} if it has a basis for its topology consisting of compact open bisections. (Equivalently, if it is \'etale and its unit space is totally disconnected.) Note that compact subsets of $\mathcal G$ need not be closed, since we do not assume that $\mathcal G$ is Hausdorff.

Let $\Bbbk$ be a field. For a (not necessarily continuous) map $f\colon \mathcal{G}\to\Bbbk$, the \emph{strict support} of $f$ is the set $f^{-1}(\Bbbk\setminus\{0\})$. We say that $f$ is \emph{compactly supported} if its strict support is contained in a compact subset of $\mathcal{G}$.

Consider two maps $f_1, f_2\colon \mathcal{G}\to\Bbbk$. Their \emph{convolution} $f_1f_2$ is defined by the formula
\[f_1f_2(g)=\sum_{g_1g_2=g}f_1(g_1)f_2(g_2).\]
If one of the maps $f_i$ is compactly supported, then the sum on the right-hand side is finite for every $g\in\mathcal{G}$, since $\mathsf{s}(g_2)=\mathsf{s}(g)$ and $\mathsf{r}(g_1)=\mathsf{r}(g)$, so there are finitely many elements $g_i$ for which the value $f_i(g_i)$ is non-zero.

In particular, the set of compactly supported functions on $\mathcal{G}$ is a $\Bbbk$-algebra with respect to the convolution operation.

\begin{df}
Let $\mathcal G$ be an \'etale groupoid.
\begin{enumerate}
\item Following Connes~\cite{Connes}, denote by $C_c(\mathcal G)$ the $\C$-algebra spanned by functions $f\colon \mathcal G\to \C$ such that there is a compact open Hausdorff subset $U\subset \mathcal G$ (which may be taken to be an open bisection) such that $f|_U$ is continuous with compact support and $f|_{\mathcal G\setminus U}=0$.
\item Let $\mathcal G$ be ample and $\Bbbk$ a field.
The algebra $\Bbbk\mathcal{G}$ is the $\Bbbk$-algebra generated by the indicator functions $1_F$ of compact open bisections $F\subset\mathcal{G}$.
\end{enumerate}
\end{df}

If $F_1, F_2$ are compact open bisections, then $F_1F_2$ is also a compact open bisection and $1_{F_1}1_{F_2}=1_{F_1F_2}$. Consequently, $\Bbbk\mathcal{G}$ is, in fact, the $\Bbbk$-linear span of the set of indicator functions of compact open bisections.

If $\mathcal{G}$ is Hausdorff, then $\Bbbk\mathcal{G}$ coincides with the algebra of all continuous (equivalently, locally constant) compactly supported functions $\mathcal{G}\to\Bbbk$, where $\Bbbk$ has discrete topology.

The algebras $\C\mathcal{G}$ and $C_c(\mathcal G)$ have the natural structure of a complex $*$-algebra with respect to the operation $f^*(g)=\overline{f(g^{-1})}$.

If $f_2\colon \mathcal{G}\to\Bbbk$ is supported on $\mathcal{G}_x$, then for every compactly supported $f_1\colon \mathcal{G}\to\Bbbk$, the function $f_1f_2$ is defined and is also supported on $\mathcal{G}_x$.

Let $x\in\mathcal{G}^{(0)}$, and consider the space $\ell^2(\mathcal{G}_x)$ of $\ell^2$-summable functions $f\colon \mathcal{G}_x\to\C$.
For an \'etale groupoid, one has that the $\ast$-representation $\lambda_x$ of $C_c(\mathcal G)$ on $\ell^2(\mathcal G_x)$ given by left multiplication (as above with $\Bbbk=\C$) is bounded since if $f$ is continuous with compact support on an open  bisection $U$, then $\|\lambda_x(f)\|\leq \|f\|_{\infty}$.  We denote by $\lambda$ the direct sum of all representations $\lambda_x$ for $x\in\mathcal{G}^{(0)}$.

\begin{df}
Let $\mathcal G$ be an  \'etale groupoid.
The \emph{reduced $C^*$-algebra} $C_r^*(\mathcal{G})$ is the completion of $C_c(\mathcal{G})$ with respect to the norm $\|f\|=\|\lambda(f)\|=\sup_{x\in\mathcal{G}^{(0)}}\|\lambda_x(f)\|$.
\end{df}

Suppose that $\mathcal G$ is ample.    Note that if
$\delta_h$ is the delta-function for the element $h\in\mathcal{G}_x$, and $F$ is a open compact $\mathcal{G}$-bisection, then $1_F\delta_h(g)=1$ if and only if $gh^{-1}\in F$, which is equivalent to $g=Fh$. Consequently, $1_F\delta_h=\delta_{Fh}$.
Therefore, multiplication by $1_F$ is a partial isometry of $\ell^2(\mathcal{G}_x)$. It follows that multiplication by elements of $\C\mathcal{G}$ from the left is a bounded $*$-representation of $\C\mathcal{G}$ on $\ell^2(\mathcal{G}_x)$. Let us denote this representation by $\lambda_x$. We denote by $\lambda$ the direct sum of all representations $\lambda_x$ for $x\in\mathcal{G}^{(0)}$.  Notice that $\lambda$ is the restriction of the corresponding representation $\lambda$ of $C_c(\mathcal G)$.   It is proved in~\cite{CZ24} that $C^*_r(\mathcal G)$ is the completion of $\C\mathcal G$ with respect to the norm $\|a\| = \|\lambda(a)\|=\sup_{x\in\mathcal{G}^{(0)}}\|\lambda_x(a)\|$, that is, $\C\mathcal G$ embeds densely and isometrically into $C_c(\mathcal G)$ with respect to the reduced norms.

For an \'etale groupoid $\mathcal{G}$, we regard elements in $C_r^*(\mathcal{G})$ as bounded functions on $\mathcal{G}$ by $a(\gamma) = \langle \lambda_x(a)\delta_x,\delta_\gamma\rangle$ where $x=\mathsf{s}(\gamma)$. This identification is sometimes called the \emph{$j$-map}, see~\cite[Proposition~II.4.2]{Renault}.  Note that $\|a\|_{\infty}\leq \|a\|$, with respect to this identification,  and, for elements of $C_c(\mathcal G)$, we obtain the original function.  If $\mathcal G$ is Hausdorff, then these functions belong to $C_0(\mathcal G)$, but  for \'etale groupoids which are not Hausdorff, there may exist non-continuous elements in $C^*_r(\mathcal{G})$ whose strict support, when regarded as functions on $\mathcal{G}$, has empty interior.
These elements are used to define what is called the
essential ideal (under the assumption that $\mathcal G$ is second countable).

\begin{df}\label{df:EssIdeal}
Let $\mathcal{G}$ be a second countable \'etale groupoid.
The \emph{essential ideal} in $C_r^*(\mathcal{G})$ is the ideal
$I_{\mathrm{ess}}(\mathcal{G})$ consisting of all elements in $C_r^*(\mathcal{G})$ whose
strict support has empty interior in $\mathcal{G}$.

If $\mathcal G$ is ample, we set $I_{\mathrm{ess}}^{\mathrm{alg}}(\mathcal{G})=I_{\mathrm{ess}}(\mathcal{G})\cap \mathbb C\mathcal{G}$, and call it the \emph{algebraic essential ideal}.
\end{df}

For an effective, minimal and second countable ample groupoid $\mathcal G$, one has that $C^*_r(\mathcal G)/I_{\mathrm{ess}}(\mathcal{G})$ is simple (\cite{KM21}) and $\C\mathcal G/I_{\mathrm{ess}}^{\mathrm{alg}}(\mathcal{G})$ is the unique simple quotient of $\C\mathcal G$ (\cite{Nekgrpd,simplicity}).  Thus $C^*_r(\mathcal G)$ is simple if and only if $I_{\mathrm{ess}}(\mathcal{G})=0$ and $\C\mathcal G$ is simple if and only if $I_{\mathrm{ess}}^{\mathrm{alg}}(\mathcal{G})=0$; see~\cite{nonhausdorffsimple}.  The case of $\C \mathcal G$ does not need second countability~\cite{simplicity}.  We remark that for \'etale groupoids that are not second countable, the definition of the essential ideal is more subtle~\cite{KM21}.

\begin{rem}
Although $I_{\mathrm{ess}}^{\mathrm{alg}}(\mathcal{G})=I_{\mathrm{ess}}(\mathcal{G})\cap \C \mathcal{G}$ by definition, it is not clear whether $I_{\mathrm{ess}}^{\mathrm{alg}}(\mathcal{G})$ is dense in $I_{\mathrm{ess}}(\mathcal{G})$. Indeed, it is conceivable that we may have
$I_{\mathrm{ess}}^{\mathrm{alg}}(\mathcal{G})=\{0\}$ while
$I_{\mathrm{ess}}(\mathcal{G})\neq \{0\}$.
\end{rem}

\section{Graded \'etale groupoids}
Let $\mathcal{G}$ be an \'etale groupoid and let
$c\colon \mathcal{G}\to \Z$ be a continuous, i.e., locally constant, cocycle.
Given $n\in\Z$, we set $\mathcal{G}_n=c^{-1}(\{n\})$,
and note that $\mathcal{G}=\bigsqcup_{n\in\Z} \mathcal{G}_n$ is a clopen partition of $\mathcal{G}$, which we call a \emph{grading} on $\mathcal{G}$. In this setting, we say that $(\mathcal{G},c)$ is a \emph{graded groupoid}. Note that
$\mathcal{G}_0$ is itself an \'etale, locally compact groupoid.

Gradings can be used to define natural circle actions on the associated groupoid C*-algebra:

\begin{df}\label{df:Gauge}
Let $(\mathcal{G},c)$ be a graded \'etale groupoid.
Given $z\in\T$, $n\in\Z$ and $f\in C_c(\mathcal{G}_n)$, we set $T_z(f)(\gamma)=z^nf(\gamma)$ for all $\gamma\in \mathcal{G}_n$. Since  $C_c(\mathcal{G})=\bigoplus_{n\in \Z} C_c(\mathcal{G}_n)$, the above formula induces a well-defined map $T_z\colon C_c(\mathcal{G})\to C_c(\mathcal{G})$, which is easily seen to be a $\ast$-algebra isomorphism.

We also write $T_z\in \mathrm{Aut}(C^*(\mathcal{G}))$ for
the extension to the universal C*-algebra of $\mathcal{G}$.
It is immediate to check that $z\mapsto T_z$ is a continuous
group action, which we call the \emph{gauge action} on $C^*(\mathcal{G})$.
\end{df}

\begin{lma}\label{lma:Gauge}
Let $\mathcal{G}$ be a graded groupoid with cocycle $c$,
let $a\in C^*(\mathcal{G})$, let $x\in\mathcal{G}^{(0)}$, let $\gamma_1,\gamma_2\in \mathcal{G}x$
and let $z\in\T$. Then
\[\langle \lambda_x(T_z(a))\delta_{\gamma_1},\delta_{\gamma_2}\rangle = z^{c(\gamma_2)-c(\gamma_1)}
 \langle \lambda_x(a)\delta_{\gamma_1},\delta_{\gamma_2}\rangle
\]
for all $\gamma_1,\gamma_2\in\mathcal{G}x$.
\end{lma}

\begin{proof}
Since $T_z$ is continuous, it suffices to show the statement assuming that $a\in C_c(\mathcal{G})$. By linearity, it suffices to assume that there exists $m\in\Z$ such that $a$ is supported on $\mathcal{G}_m$.

Since for $f\in C_c(\mathcal{G})$ one has $\langle \lambda(f)\delta_{\gamma_1},\gamma_2\rangle=f(\gamma_2\gamma_1^{-1})$, in
order to establish the formula in the statement we need to
show that
\[z^ma(\gamma_2\gamma_1^{-1})=z^{c(\gamma_2)-c(\gamma_1)}a(\gamma_2\gamma_1^{-1}).\]
Note that both expressions are zero unless $c(\gamma_2\gamma_1^{-1})=m$. Since $c(\gamma_2\gamma_1^{-1})=c(\gamma_2)-c(\gamma_1)$, the
identity follows.
\end{proof}

It follows from the lemma that the gauge action of $\mathbb T$ on $C^*(\mathcal G)$ induces a well-defined gauge action of $\mathbb T$ on the reduced $C^*$-algebra $C^*_r(\mathcal G)$.

Recall that if $T\colon \T\to \Aut(A)$ is a circle action on
a $C^*$-algebra $A$, then there is a canonical faithful condition expectation $E\colon A\to A^{\mathbb T}$ given by $E(a)=\int_\T T_z(a)\, dz$, where the integration is with respect to the normalized Haar measure.  If $\mathcal G$ is an \'etale groupoid and $\mathcal H$ is an open subgroupoid, then $C_c(\mathcal H)$ embeds in $C_c(\mathcal G)$ via extension by zero and the closure of $C_c(\mathcal H)$ in $C^*_r(\mathcal G)$ is isomorphic to $C^*_r(\mathcal H)$; see~\cite[Proposition~1.9]{PhillipsCross}, which does not use the blanket Hausdorff assumption.

\begin{prop}\label{Prop:FixedPt}
Let $(\mathcal{G},c)$ be a graded groupoid, and let $T\colon \T\to \Aut(C^*(\mathcal{G}))$ be the gauge action defined in
\autoref{df:Gauge}. Then
\[C^*_r(\mathcal{G})^\T=C^*_r(\mathcal{G}_0).\]
In particular, any element in $C^*_r(\mathcal G)^{\mathbb T}\subseteq C^*_r(\mathcal G)$ is supported on $\mathcal G_0$.
\end{prop}
\begin{proof}
We identify $C^*_r(\mathcal G_0)$ with $\overline{C_c(\mathcal G_0)}\subseteq C^*_r(\mathcal G)$.
Given $m\in\Z$ and $f\in C_c(\mathcal{G}_m)$, it follows from \autoref{lma:Gauge} that
\[E(f) = \begin{cases}f, & \text{if}\ m=0\\ 0, & \text{else.}\end{cases}.\]
In particular, we deduce that $C_c(\mathcal G_0)\subseteq C^*_r(\mathcal G)^{\mathbb T}$ and hence $C_r^*(\mathcal G_0)\subseteq C^*_r(\mathcal G)^{\mathbb T}$.  For the converse inclusion, given $a\in  C^*_r(\mathcal G)^{\mathbb T}$, find
a sequence $(f_n)_{n\in\N}$ in $C_c(\mathcal{G})$ which converges
to $a$ in $C_r^*(\mathcal{G})$. Then $a=E(a) = \lim E(f_n)$ and $E(f_n)\in C_c(\mathcal G_0)$ by the above formula, whence $a\in C^*_r(\mathcal G_0)$.  
\end{proof}

Next, we show that for,  graded groupoids, the essential ideal is always gauge-invariant.

\begin{lma}\label{lma:IessInv}
Let $\mathcal G$ be a second countable graded \'etale groupoid. Then,
for every $z\in\T$, we have $T_z(I_{\mathrm{ess}}(\mathcal G))= I_{\mathrm{ess}}(\mathcal G)$. Moreover,
\[E(I_{\mathrm{ess}}(\mathcal G))=I_{\mathrm{ess}}(\mathcal G_0)=I_{\mathrm{ess}}(\mathcal G)\cap C^*(\mathcal G)^{\mathbb T}.\]
\end{lma}
\begin{proof}
It follows from \autoref{lma:Gauge} that the strict support of $a$ and $T_z(a)$ coincide, and hence $a\in I_{\mathrm{ess}}(\mathcal G)$ if and only if $T_z(a)\in I_{\mathrm{ess}}(\mathcal G)$.  

For the last statement, since elements of $C^*_r(\mathcal G)^{\mathbb T}=C_r^*(\mathcal G_0)$ are supported on $\mathcal G_0$ by \autoref{Prop:FixedPt}, and as $\mathcal{G}_0$ is clopen in $\mathcal{G}$, it follows that $I_{\mathrm{ess}}(\mathcal G_0) = I_{\mathrm{ess}}(\mathcal G)\cap C_r^*(\mathcal G_0)$.  Therefore, $I_{\mathrm{ess}}(\mathcal G_0)\subseteq E(I_{\mathrm{ess}}(\mathcal G))$.  But the first statement
together with
the fact that $E$ is given by $E(a)=\int_\T T_z(a) dz$ for all $a\in C_r^*(\mathcal{G})$, with integration with respect to the normalized Haar measure, implies that $E(I_{\mathrm{ess}}(\mathcal G))\subseteq I_{\mathrm{ess}}(\mathcal G)\cap C_r^*(\mathcal G)^{\mathbb T} = I_{\mathrm{ess}}(\mathcal G)\cap C_r^*(\mathcal G_0)=I_{\mathrm{ess}}(\mathcal G_0)$, and the second statement follows.
\end{proof}

\section{Preliminaries on contracting self-similar groups}

Given a set $\X$, which we regard as an alphabet,
we write $\X^\omega$ for the set of all infinite words $\mathsf{x}_1\mathsf{x}_2\cdots$ with $\mathsf{x}_n\in\X$ for all $n\in\N$. Similarly, we
let $\X^\ast$ (respectively, $\X^+$) denote the set of all finite (respectively, finite non-empty) words on $\X$.  We equip $\X^\omega$ with the metric $d(u,v) = 2^{-|u\wedge v|}$ where $u\wedge v$ is the longest common prefix of $u,v$.

Since every closed ball in $\X^\omega$ is equal to the set $u\X^\omega$ of sequences with a given prefix $u\in\X^\omega$, every isometry $f$ of $\X^\omega$ induces a permutation of $\X^\ast$. This permutation preserves inclusion of balls, hence it is an automorphism of the rooted tree with the set of vertices $\X^\ast$, and edges of the form $\{v, vx\}$ for $v\in\X^\ast$ and $x\in\X$, i.e., the Cayley graph of $\X^*$.

\begin{df}
Let $G$ be a discrete group, let $\X$ be a finite set and let $G\curvearrowright \X^\omega$ be a faithful action by isometries.  There is the induced action of $G$ on $\X^*$.

We say that $(G,\X)$ is \emph{self-similar}
if, for all $g\in G$ and all $\mathsf{x}\in \X$, there exists $g|_\mathsf{x}\in G$
such that for all $w\in \X^\omega$,
we have $g(\mathsf{x}w)=g(\mathsf{x})g|_\mathsf{x}(w)$.  We can define inductively $g|_w$ for $w\in \X^*$ by $g|_{\mathsf xu}=(g|_{\mathsf x})|_u$, and  $g|_{\emptyset}=g$ for the empty word. Then $g(vu)=g(v)g|_v(u)$ for all $v, u\in\X^*$.

We moreover say that $(G,\X)$ is \emph{contracting} if there exists
a finite subset $N\subseteq G$ with the following properties:
\be\item For every $n\in N$ and for all $\mathsf{x}\in\X$,
we have $n|_{\mathsf{x}}\in N$.
\item For all $g\in G$ there exists $m\geq 0$ such that
$g|_w\in N$ for all $w\in \X^m$, that is, $g|_{\X^m}\subseteq N$. \ee

If $(G,\X)$ is contracting and self-similar, then there is a smallest finite subset $\mathcal{N}\subseteq G$ satisfying properties (a) and (b) above called the \emph{nucleus} of $(G,\X)$.
\end{df}

Note that the elements $g|_{\mathsf x}$ in the definition above must be unique, since the action $G\curvearrowright \X^\omega$ is assumed to be faithful.  We refer the reader to~\cite{selfsimilar} for more details on self-similar groups.

The second named author defined $C^*$-algebras of self-similar groups in~\cite{Nek:jot,Nekcstar} and their $*$-algebraic analogues in~\cite{Nekgrpd}. See also~\cite{ExelPardoSelf} for generalizations.

\begin{df}
Let $(G,\X)$ be a self-similar group. We define the $\ast$-algebra $\C(G,\X)$ to be the universal
$\ast$-algebra
generated by $\{\mathsf{u}_g\colon g\in G\} \cup \{s_\mathsf{x}\colon \mathsf{x}\in \X\}$, subject to the relations
\be\item[(1)] the map $g\mapsto \mathsf{u}_g$ is a unitary representation of $G$;
\item [(2)] $\mathsf{u}_g s_{\mathsf{x}}=s_{g(\mathsf{x})}\mathsf{u}_{g|_{\mathsf{x}}}$ for all $g\in G$ and all $\mathsf{x}\in\X$;
\item[(CK1)] $s_{\mathsf{x}}^*s_{\mathsf{y}}=\delta_{\mathsf{x},\mathsf{y}}$ for all $\mathsf{x},\mathsf{y}\in\X$;
\item[(CK2)] $\sum_{\mathsf{x}\in\X}s_{\mathsf{x}}s_{\mathsf{x}}^*=1$.
\ee

We also define the C*-algebra $\Oo(G,\X)$ of $(G,\X)$
to be the universal C*-algebra with the above generators and relations.
\end{df}

In the case that $(G,\X)$ is contracting, it is known that $\C(G,\X)$ and $\mathcal O(G,\X)$ are finitely presented and depend only on the nucleus $\mathcal N$ (as a set of isometries of $\X^\omega$) \cite{Nekcstar,SS23}. 

\begin{eg} When $G$ is the trivial group, the algebra $\C(G,\X)$ is the Leavitt algebra $L_{|\X|}$, and the C*-algebra $\Oo(G,\X)$ is the Cuntz algebra $\mathcal{O}_{|\X|}$.
\end{eg}

The algebras defined above admit natural
ample groupoid models, which
we proceed to describe. Let $S(G,\X)$ be the $\ast$-semigroup with zero given by the generators and the relations (1), (2) and (CK1).  Then $S(G,\X)$ is an inverse semigroup and each of its nonzero elements can be uniquely expressed in the form $s_v\mathsf{u}_gs_w^*$ with $v,w\in\X^\ast$. We let $S(G,\X)$ act on $\X^\omega$ canonically, where the domain of
$s_v\mathsf{u}_gs_w^*$ is the clopen set $w\X^\omega$,
and the corresponding partial homeomorphism maps $w\rho$
to $v(g(\rho))$.

\begin{df}
Given a self-similar group $(G,\X)$, we let $\mathcal{G}(G,\X)$
be the groupoid of germs of the inverse semigroup action $S(G,\X)\curvearrowright \X^\omega$ described above.
\end{df}

It is clear from the construction that $\mathcal{G}(G,\X)$ is
ample, minimal and effective~\cite{Nekcstar,ExelPardoSelf}.  It is second countable if $G$ is countable.
It may, however, fail to be Hausdorff. Moreover,
we have $\mathbb C\mathcal{G}(G,\X)\cong \C(G,\X)$ and
$C^*(\mathcal{G}(G,\X))\cong \Oo(G,\X)$, both canonically~\cite{Nekcstar,ExelPardoSelf,SS23}.

The isomorphism maps $\mathsf{u}_g$ and $s_{\mathsf{x}}$ to the characteristic functions of the sets of germs of the transformations $w\mapsto g(w)$ and $w\mapsto\mathsf{x}w$, respectively. We will identify the generators with the corresponding elements of the algebra.

When $(G,\X)$ is contracting it is known that $\mathcal(G,\X)$ is amenable, and hence $\Oo(G,X)\cong C_r^*(\mathcal G(G,\X))$~\cite{Nekcstar}. (The argument of~\cite{Nekcstar}, given for the Hausdorff case only, extends to the general case by the results of J.~Renault~\cite{RenaultnonH}; also, the self-replicating hypothesis in~\cite{Nekcstar} is unnecessary, cf.~\cite{MillerStein}).

There is also a canonical grading on $\mathcal{G}(G,\X)$,
which gives rise to a gauge action by the discussion in
the previous section.

\begin{df}
Given a self-similar group $(G,\X)$, we let $c\colon
\mathcal{G}(G,\X)\to\Z$ be the continuous cocycle given by
$c([s_v\mathsf{u}_gs_w^*,\rho])=|v|-|w|$, for all germs
$[s_v\mathsf{u}_gs_w^*,\rho]\in \mathcal{G}(G,\X)$.
\end{df}

\begin{rem}
Under the natural identification of $C^*(\mathcal{G}(G,\X))$ with
$\mathcal{O}(G,\X)$, one can check that the gauge action from
\autoref{df:Gauge} is given by
$T_z(\mathsf{u}_g)=\mathsf{u}_g$ and $T_z(s_\mathsf{x})=zs_\mathsf{x}$ for all $z\in\T$, all $g\in G$ and all $x\in\X$.  Also, identifying $\mathbb C(G,\X)$ with $\mathbb C\mathcal{G}(G,\X)$, the element $s_u\mathsf{u}_gs_v^*$ is homogeneous of degree $|u|-|v|$.
\end{rem}

The subalgebra $A=\mathcal{O}(G,\X)^\T$
of $\mathcal{O}(G,\X)$   consisting
of gauge-invariant elements, and its algebraic analogue the homogeneous component $A_0=\mathbb C(G,\X)_0$, will play an important role for us.  Note that $A_0$ and $A$ are just the complex and $C^*$-algebras of the ample groupoid $\mathcal G(G,\X)_0$, and hence $A_0$ is dense in $A$.

\begin{rem}\label{rem:IndLimitA}
The algebras $A$ and $A_0$ admit the following descriptions, due to the second named author.  We shall not formally use these descriptions but they explain the main idea well.
We set $d=|\X|$, and recall that there is a canonical
unitary representation of $G$ on $M_d(C^*(G))$
given by  \[\mu(g)_{\mathsf x,\mathsf y} = \begin{cases} g|_{\mathsf y}, & \text{if}\ \mathsf x=g(\mathsf y), \\
                                            0, & \text{else.}
                                           \end{cases}
\]
for all $g\in G$. By integrating this representation to a
homomorphism $\mu\colon C^*(G)\to M_d(C^*(G))$,
and tensoring this map with the identity on suitable matrix
algebras, we get an inductive system
\[C^*(G)\to M_d(C^*(G))\to M_{d^2}(C^*(G))\to \cdots,\]
and  $A$ is the $C^*$-algebraic direct limit. Also, $A_0$ is the $\ast$-algebraic direct limit of
\[\C G\to M_d(\C G)\to M_{d^2}(\C G)\to \cdots,\]
which is a dense subalgebra of $A$.
\end{rem}

Suppose that $(G,\X)$ is contracting with nucleus $\mathcal N$ and write $\mathbb C\mathcal N$ for the span of $\mathcal N$ in $\C G$.  Adopt the notation used in the remark above.
If $a\in \mathbb C\mathcal N$, then $\mu(a)\in M_d(\mathbb C\mathcal N)$ since $\mathcal N$ is closed under the operation $g\mapsto g|_{\mathsf x}$.  Moreover, if $a\in \mathbb CG$, then since $a$ has finite support, by the definition of contraction, we can find an integer $n$ with the image of $a$ in $M_{d^n}(\mathbb C\mathcal N)$.  Thus $A_0$ is the direct limit of the inductive system of vector spaces
\[\C\mathcal N\to M_d(\C\mathcal N)\to M_{d^2}(\C\mathcal N)\to \cdots.\]
The map $\C\mathcal{N}\longrightarrow\mathcal{O}(G, \X)$, and consequently the linear maps in the sequence above, are not injective in general. From now on, we will denote by $\C\mathcal{N}$ the image of the span of $\mathcal{N}$ in $\mathcal{O}(G, \X)$.

In this notation the maps in the sequence above are injective, and $A_0$ is the corresponding increasing union of the subspaces $M_{d^n}(\C\mathcal{N})$.
The intuition is then that, for any $a\in A$, the entries of $a$, as you move along the inductive system, accumulate in $\mathbb C\mathcal N$.


\begin{prop}
The essential ideals of $I_{\mathrm{ess}}^{\mathrm{alg}}(\mathcal{G}(G, \X))$ and $I_{\mathrm{ess}}^{\mathrm{alg}}(\mathcal{G}(G, \X)_0)$ are generated in the respective algebras by the intersection of the essential ideal with $\C\mathcal{N}$.
\end{prop}

\begin{proof}
The fibers $c^{-1}(n)$ of the cocycle $c\colon \mathcal{G}(G, \X)\to\Z$ are pairwise disjoint and open. Consequently, an element $f\in\C(G, \X)$ belongs to the algebraic essential ideal $I_{\mathrm{ess}}^{\mathrm{alg}}(\mathcal G(G,\X))$ if and only if all its restrictions to the fibers $c^{-1}(n)$ belong to it (as only finitely many of them will be non-zero).

Similarly, if $n>0$, then the sets $s_v\cdot\mathcal{G}(G, \X)_0$, for $v\in\X^n$, of germs of transformations of the form $s_{vu_1}\mathsf{u}_gs_{u_2}^*$ for $|u_1|=|u_2|$ are also open and disjoint. Consequently, an element of $\C(G, \X)$ supported on $c^{-1}(n)$ belongs to the essential ideal if and only if it is equal to a sum $\sum_{v\in\X^n}s_vf_v$ for $f_v$ belonging to the intersection of the essential ideal with $\C(G, \X)_0$.

Similar arguments show that that an element $f\in\C(G, \X)$ supported on $c^{-1}(n)$ for $n<0$ belongs to the essential ideal if and only if it is $f=\sum_{v\in\X^{n}}f_vs_v^*$ for $f_v$ belonging to the intersection of the essential ideal with $\C(G, \X)_0$.

This proves that the essential ideal of $\C(G, \X)$ is generated in $\C(G, \X)$ by the essential ideal of $A_0=\C\mathcal{G}(G, \X)_0$.

For every element $f\in A_0$ there exists $n$ such that $f=\sum_{u, v\in\X^n}s_uf_{u, v}s_v^*$ for some $f_{u, v}\in\C\mathcal{N}$. Since $f_{u, v}=s_u^*fs_v$, the element $f$ belongs to the essential ideal if and only if all elements $f_{u, v}$ belong to it. This finishes the proof of the proposition.
\end{proof}

\begin{cor}
The algebra $\C(G, \X)/I_{\mathrm{ess}}^{\mathrm{alg}}(\mathcal{G}(G, \X))$ is finitely presented.
\end{cor}

Note that $\mathbb C\mathcal N\cap I_{\mathrm{ess}}^{\mathrm{alg}}(\mathcal{G}(G, \X))$ is described explicitly in~\cite{SS23}.

\section{Simplicity of $\Oo(G,\X)$ is equivalent
to simplicity of $\C(G,\X)$}

For a self-similar group $(G,\X)$, we write
$I_{\mathrm{ess}}^{\mathrm{alg}}(G,\X)$ and
$I_{\mathrm{ess}}(G,\X)$ for the essential ideals in
the respective groupoid algebras of $\mathcal{G}(G,\X)$.
Since $\mathcal G(G,\X)$ is minimal and effective, it follows that $\C(G,\X)$ is simple if and only if
$I_{\mathrm{ess}}^{\mathrm{alg}}(G,\X)=\{0\}$, and
$\Oo(G,\X)$ is simple if and only if
$I_{\mathrm{ess}}(G,\X)=\{0\}$, the latter assuming that $G$ is either countable and amenable or contracting.

Since
$I_{\mathrm{ess}}^{\mathrm{alg}}(G,\X)= I_{\mathrm{ess}}(G,\X)\cap \C(G,\X)$, it follows that if $\C(G,\X)$ is not simple, then
neither is $\mathcal{O}(G,\X)$. In this section, we want to prove the
converse for contracting groups, namely that if $\mathcal{O}(G,\X)$ is not simple, then
neither is $\C(G,\X)$.

Note that $\C \mathcal N\subset\mathcal O(G,\X)$ is a finite-dimensional vector space, which we regard as a normed vector space with the norm inherited from its inclusion in $\mathcal O(G,\X)$.

\begin{lma}\label{lma:distNucleus}
Let $(G,\X)$ be a contracting self-similar group
with nucleus $\mathcal N$, and use the notation from \autoref{rem:IndLimitA}.
Given $a\in A$ and $\ep>0$, there exists $n_0\in\N$ such that
\[\dist(s_v^*as_w,\C \mathcal N)<\ep\]
for all $v,w\in \X^n$ with $n\geq n_0$.
\end{lma}
\begin{proof}
As $A_0$ is dense in $A$, we may find $b\in A_0$ with $\|a-b\|<\ep$. Find scalars $\alpha_1,\ldots,\alpha_k\in\C$, group elements $g_1,\ldots,g_k\in G$ and finite words $w_1,\ldots,w_k, v_1,\ldots, v_k$ with $|w_j|=|v_j|$ such that $b=\sum_{j=1}^k \alpha_j s_{w_j}\mathsf{u}_{g_j} s_{v_j}^*$. 
We claim that we can choose the words so that $|w_1|=\cdots=|w_k|$. Indeed, assume without loss of generality that $|w_1|\geq |w_j|$ for all $j=2,\ldots,k$ and
fix such a $j$. By (CK2), we have
\[s_{w_j}\mathsf{u}_{g_j}s_{v_j}^*=s_{w_j}\mathsf{u}_{g_j} \sum_{|w|=|w_1|-|w_j|} s_ws_w^* s_{v_j}^*=\sum_{|\mu|=|w_1|-|w_j|} s_{w_j}s_{ g_j(\mu)}\mathsf{u}_{g_j|_{\mu}} s_w^*s_{v_j}^*,\]
and observe that $w_j g_j(\mu)$ and $v_j\mu$ have the same length as $w_1$. An iteration of this procedure then establishes the claim.

Since $G$ is contracting, there is $m\geq 0$ such that $g_j|_{\X^m}$
belongs to $\mathcal{N}$ for all $j=1,\ldots,k$. Set $n_0=|w_1|+m$.
Let $n\geq n_0$ and let $v,w\in \X^n$.
Then $s_v^*bs_w$ belongs to $\C \mathcal N$,
and moreover
\[\|s_v^*as_w-s_v^*bs_w\|\leq \|s_v\|\|a-b\|\|s_w\|<\ep,\]
and thus $\dist(s_v^*as_w,\C \mathcal N)<\ep$ as desired.
\end{proof}

\begin{thm}\label{t:main.result}
Let $(G,\X)$ be a contracting self-similar group.
Then $\C(G,\X)$ is simple if and only if
$\Oo(G,\X)$ is simple, if and only if $\C\mathcal{N}\cap I_{\mathrm{ess}}^{\mathrm{alg}}(G, \X)\neq 0$.
\end{thm}
\begin{proof}
It suffices to show that if $\mathcal{O}(G,\X)$ is not simple, then
$\C\mathcal{N}\cap I_{\mathrm{ess}}^{\mathrm{alg}}(G,\X)\neq 0$.
The proof consists in producing sequences $(a_n)_{n\in\N}$ and
$(c_n)_{n\in\N}$ with the following properties:
\be
\item $a_n\in I_{\mathrm{ess}}(G,\X)$ for all $n\in\N$;
\item $c_n\in \C \mathcal N$ for all $n\in\N$;
\item $\|a_n-c_n\|<\frac{1}{n}$ for all $n\in\N$;
\item $\varepsilon_0\leq \|a_n\|\leq \|a\|$ for some $\varepsilon_0>0$.\ee

Assuming that sequences as above exist, let us show that
$\C\mathcal N\cap I_{\mathrm{ess}}^{\mathrm{alg}}(G,\X)$ is nontrivial.
Since $\C \mathcal N$ is finite dimensional, it follows from
conditions (2), (3) and (4) above that there is a
nonzero accumulation point $c\in \C \mathcal N\subseteq \C(G,\X)$
for $(c_n)_{n\in\N}$. By (3), $c$
is also an accumulation point for $(a_n)_{n\in\N}$, and therefore
belongs to $I_{\mathrm{ess}}(G,\X)$ by (1). Thus $c$ is a nonzero
element in $\C\mathcal N\cap I_{\mathrm{ess}}(G,\X)\cap \C(G,\X)= \C\mathcal N\cap I_{\mathrm{ess}}^{\mathrm{alg}}(G,\X)$, as desired.

Now we construct those sequences, assuming
that $\mathcal{O}(G,\X)$ is not simple. Find
a nonzero positive element $a\in I_{\mathrm{ess}}(G,\X)$.
Then $E(a)\in I_{\mathrm{ess}}(G,\X)$ by \autoref{lma:IessInv},
and $E(a)\neq 0$ because $E$ is faithful.
Thus, upon replacing
$a$ with $E(a)$, we may assume that $a\in I_{\mathrm{ess}}(G,\X)\cap A$.
Since $\lambda(a)\neq 0$, we can find a unit $x\in \mathcal G(G,\X)^0$ and $\gamma\in \mathcal G(G,\X)_x$ with $a(\gamma)=\langle \lambda_x(a)\delta_x,\delta_{\gamma}\rangle\neq 0$, and
set $\varepsilon_0=|a(\gamma)|$. Let $x=[1,\rho]\in \mathcal{G}(G,\X)$ with $\rho\in \X^{\omega}$ and $\gamma=[s_u\mathsf{u}_gs_v^*,\rho]$ a germ. In particular, note that $\rho=v\rho'$.

Given $n\geq 1$, use \autoref{lma:distNucleus} to find $m$ such that $d(s_w^*as_z,\mathbb C\mathcal{N})<\frac{1}{n}$ for all words $w,z$ of length at least $m$.
Let $v_m$ be the prefix of length $m$ of $\rho'$ and set $u_m =g(v_m)$, and note that $|u_m|=|v_m|$.
Define $a_n=s_{uu_m}^*as_{vv_m}$, which belongs to $I_{\mathrm{ess}}(G,\X)$ because $I_{\mathrm{ess}}(G,\X)$ is an ideal. Since $|uu_m|=|vv_m|\geq m$, there exists $c_n\in \C \mathcal N$
such that $\|a_n-c_n\|<\tfrac{1}{n}$.  Note that
\[[s_u\mathsf{u}_gs_v^*,\rho]= [s_u\mathsf{u}_gs_v^*s_{vv_m}s_{vv_m}^*,\rho] = [s_u\mathsf{u}_gs_{v_m}s_{vv_m}^*,\rho] = [s_{uu_m}\mathsf{u}_{g|_{v_m}}s_{vv_m}^*,\rho].\]
Therefore, we
we compute
\begin{align*}
\langle \lambda_x(a_n)\delta_{[s_{vv_m}^*,\rho]},\delta_{[\mathsf u_{g|_{v_m}}s_{vv_m}^*,\rho]}\rangle &= \langle \lambda_x(s_{uu_m})^*\lambda_x(a)\lambda_x(s_{vv_m})\delta_{[s_{vv_m}^*,\rho]},\delta_{[\mathsf u_{g|_{v_m}}s_{vv_m}^*,\rho]}\rangle\\ &=
\langle \lambda_x(a)\delta_{[1,\rho]},\lambda_x(s_{uu_m})\delta_{[u_{g|_{v_m}}s_{vv_m}^*,\rho]}\rangle \\ &= \langle \lambda_x(a)\delta_{[1,\rho]},\delta_{[s_{uu_m}\mathsf u_{g|_{v_m}}s_{vv_m}^*,\rho]}\rangle =\varepsilon_0.
\end{align*}
Using that $s_u$ and $s_v$ are isometries at the last step, the Cauchy-Schwarz lemma then implies that
\begin{align*}\varepsilon_0 &= |\langle \lambda_x(a_n)\delta_{[s_{vv_m}^*,\rho]},\delta_{[\mathsf u_{g|_{v_m}}s_{vv_m}^*,\rho]}\rangle|\\
&\leq \|\lambda_x(a_n)\delta_{[s_{vv_m}^*,\rho]}\|\cdot \|\delta_{[\mathsf u_{g|_{v_m}}s_{vv_m}^*,\rho]}\|\\
&\leq \|a_n\|\leq \|s_u^*\|\|a\|\|s_v\|=\|a\|.\end{align*}
This completes the proof.
\end{proof}

A special case of \autoref{t:main.result} was proved by Yoshida~\cite{2021arXiv210202199Y} for the class of multispinal groups introduced by the third author and Szaka\'cs in~\cite{SS23}.  An explicit criterion for simplicity of $\C(G,\X)$ was given in~\cite{SS23} for multispinal groups, and it was shown in~\cite{2021arXiv210202199Y} that the same criterion holds for simplicity of $\Oo(G,X)$.

\section{A simplification of the simplicity algorithm}
In~\cite{SS23}, the third author and Szaka\'cs gave an algorithm to determine if $\mathbb C(G,\X)$ is simple given the nucleus as input.  We give here a simplification of the algorithm that runs in exponential time in the size of the nucleus, which is a significant improvement on the runtime in~\cite{SS23}.
Let $(G,\X)$ be a contracting self-similar group with nucleus $\mathcal N$. It is shown in~\cite[Corollary~5.8]{SS23} that if $I_{\mathrm{ess}}^{\mathrm{alg}}(G,\X)\neq 0$, then there is an element $a\in \C\mathcal N$ that represents a nonzero element of $I_{\mathrm{ess}}^{\mathrm{alg}}(G,\X)$. We will narrow this down to elements supported on certain finite subgroups of $\mathcal N$.

Associated to the nucleus is a finite automaton with state set $\mathcal N$ and alphabet $\X$.  The edges are of the from $g\xrightarrow{\mathsf x\mid g(\mathsf x)}g|_{\mathsf x}$.  Given $w\in \X^*$ and $g\in \mathcal N$, there is a unique path starting at $g$ whose left-hand side is labeled by $w$.  This path ends at $g|_w$ and the right-hand side of the label of the path is $g(w)$.  Consider the subgraph $\mathcal H$ obtained by removing all edges save those of the form $g\xrightarrow{\mathsf x\mid \mathsf x}g|_{\mathsf x}$ (which we now view as being labeled by $\mathsf x$).  Let $\mathcal C\subseteq \mathcal N$ consist of those elements on some nonempty directed cycle of $\mathcal H$.  Note that if $g\in \mathcal C$ and $\mathsf{x}$ labels an edge of a directed cycle of $\mathcal H$ at $g$, then $g|_{\mathsf x}\in \mathcal C$.  
Let us write $e$ for the identity of $G$.

\begin{prop}\label{p:cyc.subs}
Let $w\in \X^+$,  and put $H_w = \{g\in G\mid g(w)=w,g|_w=g\}$.  Then $H_w$ is a finite subgroup contained in $\mathcal C\subseteq \mathcal N$.
\end{prop}
\begin{proof}
Note that $H_w$ is a subgroup. Indeed, trivially $e\in H_w$, and if $g,h\in H_w$, then $gh(w)=w$ and $(gh)|_w= g|_{h(w)}h|_w =g|_wh|_w= gh$, and so $gh\in H_w$.  Finally, if $g\in H_w$, then $g^{-1}(w)=w$ and $g^{-1} |_w = (g|_{g^{-1}(w)})^{-1} = g^{-1}$.

Observe that if $g\in H_w$, then $g\in \mathcal N$ as $g|_{w^n}=g$ for all $n\geq 0$.
Moreover, $w$ labels a nonempty closed path in $\mathcal H$ at $g$, and hence $g$ lies on a directed cycle, that is, $g\in \mathcal C$.
\end{proof}

Note $g\in H_w$ if and only if $g\in \mathcal C$, and $w$ labels a closed path in $\mathcal H$ at $g$.

It turns out that every element of $\C H_w$ is nontrivial in $\C(G,\X)$.

\begin{lma}\label{l:cycle.groups}
Let $0\neq a\in \C H_w$ with $w\in \X^+$.  Then $a$ is nonzero in $\C(G,\X)$.
\end{lma}
\begin{proof}
Note that if $g\neq h\in H_w$, then $[\mathsf u_g,w^{\infty}]\neq [\mathsf u_h,w^{\infty}]$ as $g|_{w^n}=g$ and $h|_{w^n}=h$ for all $n\geq 0$.  If follows that if $a=\sum_{g\in H_w}a_gg$ with $a_h\neq 0$, then the function on $\mathcal G(G,X)$ defined by $a$ does not vanish on $[h,w^{\infty}]$, and hence $a\neq 0$ in $\C(G,\X)$.
\end{proof}

We now give an algorithm that determines if a subset of $\mathcal C$ is contained in $H_w$ for some $w\in \X^+$.  By definition, every singleton belongs to such a subgroup.

\begin{lma}\label{l:cycle.decide}
One can find in time $O(|\X|\cdot 2^{|\mathcal C|})$ all subsets of $\mathcal C$ that are contained in some subgroup of the form $H_w$ with $w\in \X^+$.
\end{lma}
\begin{proof}
Define a labeled digraph with vertex set $2^{\mathcal C}$.  There is a directed edge labeled by $\mathsf x\in \X$ from $A$ to $B$ if $|A|=|B|$ and there is an edge labeled by $\mathsf x$ in $\mathcal H$ from each element of $A$ to a distinct element of $B$.  One can construct this digraph in time $O(|\X|\cdot 2^{|\mathcal C|})$ and it has $2^{|\mathcal C|}$ vertices and at most $|X|\cdot 2^{|\mathcal C|}$ edges.   Observe that $w$ labels a closed path at a subset $Y$ if and only if $Y\subseteq H_{w^{|Y|!}}$.  Thus $Y\subseteq H_u$ for some $u\in \X^+$ if and only if $Y$ lies on a directed cycle. The set of such vertices can be found in time linear in the size of the digraph using Tarjan's algorithm.
\end{proof}

Now we show that if the algebraic essential ideal is nonzero, then it contains an element supported on some subgroup $H_w$.

\begin{prop}\label{p:reduce.to.subgroup}
If $I_{\mathrm{ess}}^{\mathrm{alg}}(G,\X)\neq 0$, then there is some $w\in X^+$  and $0\neq a\in \C H_w$ with  $a\in I_{\mathrm{ess}}^{\mathrm{alg}}(G,\X)$.
\end{prop}
\begin{proof}
It is known~\cite[Proposition~5.7]{SS23} that if $I_{\mathrm{ess}}^{\mathrm{alg}}(G,\X)\neq 0$, then there is a nonzero element $a\in \C\mathcal N\cap I_{\mathrm{ess}}^{\mathrm{alg}}(G,\X)$.  Without loss of generality we may assume that $a=\sum_{g\in \mathcal N}a_g\mathsf{u}_g$  with as few $a_g\neq 0$ as possible.  Then $s_{\mathsf y}^*as_{\mathsf x}\in \C\mathcal N$ and represents an element of $I_{\mathrm{ess}}^{\mathrm{alg}}(G,\X)$ for all $\mathsf x,\mathsf y\in \X$, and in $\C(G,\X)$ we have $a=\sum_{\mathsf x,\mathsf y\in \X}s_{\mathsf y}(s_{\mathsf y}^*as_{\mathsf x})s_{\mathsf x^*}$.  Moreover, if  $s_{\mathsf y}^*as_{\mathsf x}\neq 0$ in $\C(G,X)$, then we must have that $g|_{\mathsf x}\neq h|_{\mathsf x}$ and $g(\mathsf x)=\mathsf y=h(\mathsf x)$ whenever  $a_g\neq 0\neq a_h$ by choice of $a$.  Since $a\neq 0$ in $\C(G,\X)$, we can find $\mathsf x_1,\mathsf y_1\in \X$ with $s_{\mathsf y_1}^*as_{\mathsf x_1}\neq 0$ and $\mathsf y_1=g(\mathsf x_1)$ for all $g$ with $a_g\neq 0$, and this element meets our minimality condition on nonzero coefficients.  Continuing in this fashion we can find infinite sequences $\xi=\mathsf x_1\mathsf x_2\cdots$ and $\upsilon=\mathsf y_1\mathsf y_2\cdots$ such that if $\xi_n = \mathsf x_1\cdots \mathsf x_n$ and $\upsilon_n=\mathsf y_1\cdots \mathsf y_n$, then $s_{\upsilon_n}^*as_{\xi_n}\neq 0$ in $\C(G,\X)$, $g|_{\xi_n}\neq h|_{\xi_n}$ whenever $g\neq h$ with $a_g\neq 0\neq a_h$ and $g(\xi_n)=\upsilon_n$ whenever $a_g\neq 0$.  Note that \[s_{\upsilon_n}^*as_{\xi_n} = \sum_{g\in \mathcal N}a_g\mathsf u_{g|_{\xi_n}}.\] Since $\mathcal N$ is finite, we can find  $n<m$ such that, for all $g\in \mathcal N$ with $a_g\neq 0$, $g|_{\xi_n} = g|_{\xi_m}$.  Let $w=\mathsf x_{n+1}\cdots \mathsf x_m$, so that $\xi_nw=\xi_m$.  Then $(g|_{\xi_n})|_w = g|_{\xi_n}$, $g(\xi_nw) =\upsilon_m=h(\xi_nw)$  and $g|_{\xi_n}\neq h|_{\xi_n}$ whenever $a_g\neq 0\neq a_h$ and $g\neq h$.  Replacing $a$ by $s_{\upsilon_n}^*as_{\xi_n}$, we may then assume that $g(w) = h(w)$ and $g|_w=g$ whenever $a_g\neq 0\neq a_h$.  Fix $h\in \mathcal N$ with $a_h\neq 0$.  Then $0\neq u_h^* a\in I_{\mathrm{ess}}^{\mathrm{alg}}(G,\X)$.  Moreover, if $a_g\neq 0$, then $h^{-1} g(w) = w$ and $(h^{-1} g)|_w = (h^{-1})|_{g(w)}g|_w = (h^{-1})|_{h(w)}g = h^{-1} g$, as $h^{-1} |_{h(w)} = (h|_w)^{-1}$.  Thus $h^{-1}g\in H_w$, and so $0\neq  u_h^*a\in \C H_w\cap I_{\mathrm{ess}}^{\mathrm{alg}}(G,\X)$.
\end{proof}

Let us now describe the algorithm from~\cite{SS23} which determines if $a\in \C\mathcal N$ represents an element of $I_{\mathrm{ess}}^{\mathrm{alg}}(G,\X)$.
 The free monoid $\X^*$ acts on the set of equivalence relations on $\mathcal N$ by the rule $a\mathrel{(\mathsf{x}\cdot {\equiv})} b$ if and only if $a(\mathsf x)=b(\mathsf x)$ and $a|_{\mathsf x}\equiv b|_{\mathsf x}$.    Let ${\equiv_w} = (w\cdot{=})$ for $w\in \X^*$, and let $\Gamma$ be the labeled Schreier digraph of the left action of $\X^*$ on  the finite set $V=\{\equiv_w\mid w\in \X^*\}$.  It is shown in~\cite[Proposition~5.10]{SS23} that $a\equiv_w b$ if and only if $\mathsf u_as_w=\mathsf u_bs_w$ in $S(G,\X)$. 

Call a vertex $v\in V$ \emph{essential} if it can be reached from a cycle in $\Gamma$ and \emph{minimal} if the strongly connected component of $v$ has no exit.  Every minimal vertex is essential, but the converse need not hold.  It turns out that $\Gamma$ has exactly one minimal strongly connected component~\cite[Corollary~5.12]{SS23}. If $v\in V$ corresponds to the equivalence relation $\equiv$, then we say that $a\in \mathbb C\mathcal N$ satisfies the equations of $v$ if $a$ maps to zero under the projection $\mathbb C\mathcal N\to \mathbb C[\mathcal N/{\equiv}]$.  Note that $a$ satisfies the equations of $\equiv_w$ if and only if $as_w=0$ in the inverse semigroup algebra $\C S(G,\X)$~\cite[Lemma~5.13]{SS23}. 

It follows from~\cite[Proposition~5.15]{SS23} that $a$ represents a nonzero element of $\C(G,\X)$ if and only if $a$ fails to satisfy the equations of some essential vertex.  On the other hand~\cite[Theorem~5.14]{SS23} says that $a$ belongs to $I_{\mathrm{ess}}^{\mathrm{alg}}(G,\X)$ if and only if it satisfies the equations of every minimal vertex.  Thus $\C(G,\X)$ is simple if and only if the rank of the integer matrix obtained from putting together all the equations of minimal vertices is the same as the rank of the matrix obtained from putting together all the equations of essential vertices.

It turns out, however, that we can replace $\Gamma$ by a quotient labeled digraph with many fewer vertices and that is easier to understand conceptually.

We say that $g\in G$ \emph{strongly fixes} $w\in \X^*$ if $g(w)=w$ and $g|_w=e$, or equivalently $\mathsf u_gs_w = s_w$. Thus the set $G_{(w)}$ of elements  strongly fixing $w$ is a subgroup of $G$ (in fact, it is the pointwise stabilizer of $w\X^{\omega}$).    Define an action of $\X^*$ on the power set of $\mathcal N$ by $\mathsf x\cdot Y =\{g\in \mathcal N\mid g(\mathsf x)=\mathsf x, g|_{\mathsf x}\in Y\}$ for $\mathsf x\in \X$.    Let $Q=\X^*\cdot \{e\}$ and put $\mathcal N_w = w\cdot \{e\}$.  Note that $Q$ is finite.   Let $\Delta$ be the directed labeled Schreier graph of the left action of $\X^*$ on $Q$.  A vertex will again be called minimal if it belongs to a strongly connected component with no exit.

\begin{prop}\label{p:meaning.of.Q}
For all $w\in \X^*$, one has $\mathcal N_w=\mathcal N\cap G_{(w)}$.  Moreover, $\mathcal N_w$ is the equivalence class of $e$ in $\equiv_w$.
\end{prop}
\begin{proof}
We proceed by induction on length.  If $|w|=0$, then since the action of $G$ is faithful, $\mathcal N_w=\{e\}=\mathcal N\cap G_{(w)}$.  Suppose the claim is true for $v$ of length $n$ and $w=\mathsf xv$ with $\mathsf x\in \X$.  Then $g\in \mathcal N_w$ if and only if $g(\mathsf x)=\mathsf x$ and $g|_{\mathsf x}\in \mathcal N_v$.  By induction this occurs if and only if $g(\mathsf x)=\mathsf x$, $g|_{\mathsf x}(v)=v$ and $(g|_{\mathsf x})|_v =e$, that is, $g\in G_{(w)}$.  This proves the first item. For the second, $g\equiv_w e$ if and only if $\mathsf u_gs_w=s_w$, that is, $g\in G_{(w)}$.
\end{proof}

It follows from  \autoref{p:meaning.of.Q} that we have a surjective label-preserving morphism of digraphs $\Phi\colon \Gamma\to \Delta$ given sending an equivalence class $\equiv$ to the $\equiv$-class of $e$.

The following result improves upon the main result of~\cite{SS23}.

\begin{thm}\label{t:simpl.main}
The following are equivalent for a contracting group $(G,\X)$.
\begin{enumerate}
  \item $\C(G,\X)$ is simple.
  \item $\Oo(G,\X)$ is simple.
  \item If $w\in X^+$ and $a\in \mathbb CH_w$ maps to $0$ under the projection $\mathbb CH_w\to \mathbb C[H_w/(H_w\cap Y)]$ for every minimal vertex $Y$ of $\Delta$ (note that $H_w\cap Y$ is a subgroup), then $a=0$ in $\C H_w$.
\end{enumerate}
Moreover, it is decidable in time exponential in the size of the nucleus whether these equivalent conditions hold, given the nucleus as input.
\end{thm}
\begin{proof}
We have already seen the equivalence of (1) and (2) in  \autoref{t:main.result}.  For the equivalence of (1) and (3), we first observe that $Y$ is a minimal vertex of $\Delta$ if and only if $Y=\Phi({\equiv})$ with $\equiv$ a minimal vertex of $\Gamma$. Indeed, suppose that $\equiv_z$ is minimal.  We must show that $\Phi({\equiv}_z)=\mathcal N_z$ (by \autoref{p:meaning.of.Q}) is minimal.   Indeed, for any $u\in \X^*$, there exists $v\in \X^*$ with $(vu\cdot {\equiv_z}) = {\equiv_z}$, and so $v(u \mathcal N_z)=\mathcal N_z$, hence $\mathcal N_z$ is minimal.  Conversely, suppose that $\mathcal N_z$ is minimal.  We can find $u\in \X^*$ such that ${(u\cdot {\equiv_z})} = {\equiv_{uz}}$ is minimal.  Then since $\mathcal N_z$ is minimal, we can find $r\in \X^*$ with $ru\mathcal N_z=\mathcal N_z$.  Then $(r\cdot {\equiv_{uz}}) = {\equiv_{ruz}}$ is minimal and $\mathcal N_z = \Phi({\equiv_{ruz}})$.

Suppose that $H\subseteq \mathcal C$ is a subgroup.  Then $H\cap \mathcal N_v = H\cap G_{(v)}$ is a subgroup of $H$ by \autoref{p:meaning.of.Q}.  Moreover, if $g,h\in H$, then we claim that $g\equiv_v h$ if and only if $g(H\cap \mathcal N_v)=h(H\cap \mathcal N_v)$.  Indeed,  $g\equiv_v h$ if and only if $\mathsf u_gs_v=\mathsf u_hs_v$, if and only if $\mathsf u_{h^{-1} g}s_v=s_v$, if and only if $h^{-1} g\equiv_v e$, that is, $h^{-1} g\in H\cap \mathcal N_v$ by \autoref{p:meaning.of.Q}. It follows that $a\in \C H$ satisfies the equations of $\equiv_v$ if and only if $a$ maps to $0$ under the projection $\C H\to \C[H/(H\cap \mathcal N_v)]$.

Suppose first that (3) holds and $0\neq I_{\mathrm{ess}}^{\mathrm{alg}}(G,\X)$.  Then by~\autoref{p:reduce.to.subgroup}, there is some $w\in X^+$  with $0\neq a\in \mathbb CH_w\cap I_{\mathrm{ess}}^{\mathrm{alg}}(G,\X)$.  Then $a$ satisfies the equations of every minimal vertex  of $\Gamma$. By the above discussion, any minimal vertex $Y$ of $\Delta$ is of the form $\mathcal N_v$ with $\equiv_v$ a minimal vertex of $\Gamma$. Since $a$ satisfies the equations of $\equiv_v$, it follows that $a$ maps to $0$ under the projection $\mathbb CH\to \mathbb C[H/(H_w\cap \mathcal N_v)]$ by the previous paragraph.  We conclude that $a=0$ by (3), a contradiction.

Next suppose that (1) holds.   By the above discussion, $a\in \C H_w$ maps to $0$ under the projection $\mathbb CH_w\to \mathbb C[H_w/(Y\cap H_w)]$ for every minimal vertex $Y$ of $\Delta$ if and only if it satisfies the equations of every minimal vertex of $\Gamma$, if and only if $a$ represents an element of $I_{\mathrm{ess}}^{\mathrm{alg}}(G,\X)$.  By  \autoref{l:cycle.groups}, if $a\neq 0$, then it represents a nonzero element of $\C(G,\X)$.  It follows from (1), that any such $a$ must be $0$ in $\C H_w$.

The decidability statement follows because $\Delta$ can be computed in exponential time in the size of the nucleus, and \autoref{l:cycle.decide} implies that we can compute in exponential time in the size of the nucleus the maximal subgroups of the form $H_w$ with $w\in X^+$, and one need only check that (3) holds for all maximal subgroups of this form.  Checking if (3) holds for this finite collection of subgroups amounts to verifying if a homogeneous system of integral equations has a nontrivial solution over $\mathbb C$.
\end{proof}

The criterion for simplicity of $\mathbb C(G,\X)$ in (3) of Theorem~\ref{t:simpl.main} is valid for simplicity of $K\mathcal G(G,\X)$ over any field $K$ since the quoted results of~\cite{SS23} hold over arbitrary fields.  In particular, one can compute in exponential time the set of all characteristics of fields for which $K\mathcal G(G,\X)$ is not simple.  Namely, checking if (3) holds for the finitely many maximal subgroups of the form $H_w$ amounts to verifying if a homogeneous system of integral equations has a nontrivial solution over $K$.  This depends only on the characteristic of $K$.  If there is a nontrivial solution over $\mathbb Q$, then since reducing modulo $p$ does not increase the rank, there is a nontrivial solution over every field.  If there is no nontrivial solution over $\mathbb Q$, then one can place the matrix in Smith normal form and if $d_1\mid \cdots \mid d_k$ are the diagonal entries, then the algebra is not simple only over fields of characteristic dividing $d_k$ (in particular, if $d_k=1$, the algebra is simple over all fields).

  It is observed in~\cite[Proposition~3.2]{SS23} that $\mathcal G(G,\X)$ is Hausdorff if and only if no nontrivial element of $\mathcal C$ strongly fixes any word (that is, no cycle in $\mathcal H$ can reach $e$ except the loops at $e$).  It follows that in the Hausdorff case $H_w\cap \mathcal N_v=\{e\}$ for any words $v,w$, and so (3) is always satisfied.

Let us demonstrate how this works on two well-known examples.  The Grigorchuk group $G$ is the self-similar group over $\X=\{0,1\}$ generated by $a,b,c,d$ with
\[\begin{array}{ll}
a(0w) = 1w, &  a(1w)=0w\cr b(0w) =0a(w), &  b(1w) =1c(w)\cr  c(0w) = 0a(w), & c(1w) = 1d(w),\cr d(0w) = 0w &  d(1w) = 1b(w).
\end{array}\]
  Then $a,b,c,d$ have order $2$ and $b,c,d,e$ form a Klein $4$-group.  The group is contracting and the nucleus is $\{e,a,b,c,d\}$.  See~\cite{selfsimilar}, for example.  The automaton $\mathcal H$ and  the graph $\Delta$ are in \autoref{fi:grig}.
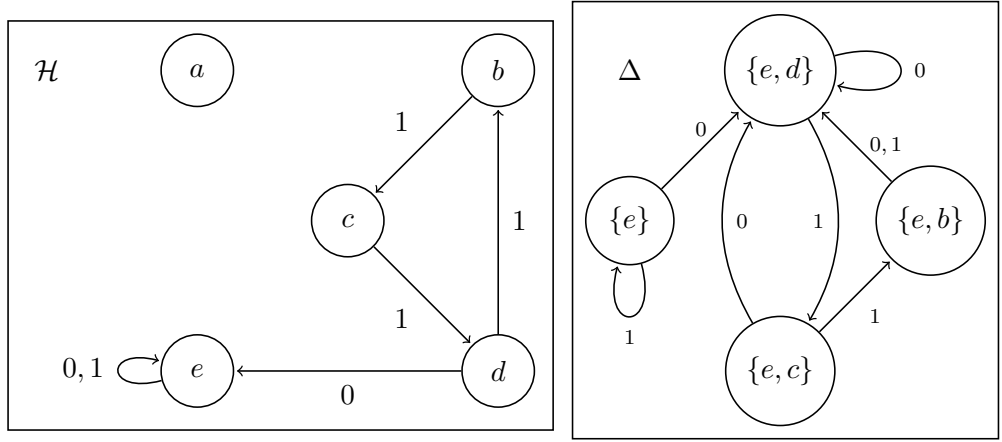
\begin{figure}[tbph]
\begin{center}
\begin{tikzpicture}[->,shorten >=1pt,%
auto,node distance=2cm,semithick,
inner sep=5pt,bend angle=30]
\begin{scope}[scale=.75, xshift=-7cm,draw, local bounding box=bounding box 0]
\node[state] (A) {$a$};
\node        (Q) [left of =A] {$\mathcal H$};
\node        (F) [right of=A]  {};
\node        (G) [below of=A]  {};
\node[state] (B) [right of=F] {$b$};
\node        (H) [below of=B]  {};
\node[state] (C) [below of=F] {$c$};
\node[state] (D) [below of=H] {$d$};
\node[state] (E) [below of=G] {$e$};
\path (B) edge  node [above left]  {$1$} (C)
      (C) edge  node [below left] {$1$} (D)
      (D) edge  node [right]  {$1$} (B)
      (D) edge  node [below]  {$0$} (E)
      (E) edge [loop left] node [left] {$0,1$} (E);
\end{scope}
\node [draw, fit=(bounding box 0)] {};
\begin{scope}[scale=.5, xshift=5cm,draw, local bounding box=bounding box 1]
\node[state] (B) {$\{e,d\}$};
\node        (G) [below of =B] {};
\node[state] (A) [left of =G]{$\{e\}$};
\node[state] (C) [below of =G] {$\{e,c\}$};
\node[state] (D) [right of=G] {$\{e,b\}$};
\node        (H) [left of =B] {$\Delta$};
\tikzstyle{every node}=[font=\scriptsize]
\path (A) edge  node [above]      {$0$} (B)
      (B) edge  [bend left]  node [left]  {$1$} (C)
      (B) edge [loop right] node [right] {$0$} (B)
      (C) edge  node [below right] {$1$} (D)
      (C) edge  [bend left] node [right] {$0$} (B)
      (D) edge  node [right]  {$0,1$} (B)
      (A) edge [loop below] node [below] {$1$} (A);
\end{scope}
\node [draw, fit=(bounding box 1)] {};
\end{tikzpicture}
\end{center}
\caption{$\mathcal H$ and $\Delta$ for the Grigorchuk group~\label{fi:grig}}
\end{figure}
Thus we see that $\mathcal C= \{e,b,c,d\}$ and that $H_{111} = \{e,b,c,d\}$.  All vertices of $\Delta$ except $\{e\}$ are minimal.  Notice that every irreducible representation of the Klein $4$-group over $K$ for a field of characteristic different than $2$ is one-dimensional and contains at least one of the three subgroups $\{e,b\}$, $\{e,c\}$ and $\{e,d\}$ in its kernel, as a subgroup of the multiplicative group of a field is cyclic.  Thus if $f\in KH_{111}$ projects to $0$ under all of the projections $K[H_{111}/N]$ with $N$ a subgroup of order $2$, then $f$ is annihilated by all irreducible representations of $H_{111}$ over $K$, and hence is $0$ as $K$ does not have characteristic $2$.  Thus $K\mathcal G(G,\X)$ is simple for any field of characteristic different than $2$ and $\Oo(G,\X)$ is simple.  On the other hand, if $K$ has characteristic $2$, then $e+b+c+d$ is annihilated under all these projections, and so $K\mathcal G(G,\X)$ is not simple in this case.  This result can be found in~\cite{nonhausdorffsimple} (see also~\cite{Nekrashevychgpd} and~\cite{SS23}).

Let us consider the Grigorchuk-Erschler group.  This group $G$ is defined over $\X=\{0,1\}$ and
generated by $a,b,c,d$ with
\[\begin{array}{ll}
a(0w) = 1w, &  a(1w)=0w\cr b(0w) =0a(w), &  b(1w) =1b(w)\cr  c(0w) = 0a(w), & c(1w) = 1d(w),\cr d(0w) = 0w &  d(1w) = 1c(w).
\end{array}\]
  Then $a,b,c,d$ have order $2$ and $b,c,d,e$ form a Klein $4$-group.  The group is contracting and the nucleus is $\{e,a,b,c,d\}$.   The automaton $\mathcal H$ and  the graph $\Delta$ are in \autoref{fi:grigersch}.
\begin{figure}[tbph]
\begin{center}
\begin{tikzpicture}[->,shorten >=1pt,%
auto,node distance=2cm,semithick,
inner sep=5pt,bend angle=30]
\begin{scope}[scale=.75, xshift=-7cm, draw, local bounding box=bounding box 0]
\node[state] (A) {$a$};
\node        (Q) [left of =A] {$\mathcal H$};
\node        (F) [right of=A]  {};
\node        (G) [below of=A]  {};
\node[state] (B) [right of=F] {$b$};
\node        (H) [below of=B]  {};
\node[state] (C) [below of=F] {$c$};
\node[state] (D) [below of=H] {$d$};
\node[state] (E) [below of=G] {$e$};
\path (B) edge [loop left] node [left]  {$1$} (B)
      (C) edge [bend left]  node [right] {$1$} (D)
      (D) edge [bend left] node [left]  {$1$} (C)
      (D) edge  node [below]  {$0$} (E)
      (E) edge [loop left] node [left] {$0,1$} (E);
\end{scope}
\node [draw, fit=(bounding box 0)] {};
\begin{scope}[scale=.5, xshift=5cm,draw, local bounding box=bounding box 1]
\node[state] (B) {$\{e,d\}$};
\node        (G) [below of =B] {};
\node[state] (A) [left of =G]{$\{e\}$};
\node[state] (C) [below of =G] {$\{e,c\}$};
\node        (H) [left of =B] {$\Delta$};
\tikzstyle{every node}=[font=\scriptsize]
\path (A) edge  node [above]      {$0$} (B)
      (B) edge  [bend left]  node [right]  {$1$} (C)
      (B) edge [loop right] node [right] {$0$} (B)
      (C) edge  [bend left] node [right] {$0,1$} (B)
      (A) edge [loop below] node [below] {$1$} (A);
\end{scope}
\node [draw, fit=(bounding box 1)] {};
\end{tikzpicture}
\end{center}
\caption{$\mathcal H$ and $\Delta$ for the Grigorchuk-Erschler group~\label{fi:grigersch}}
\end{figure}
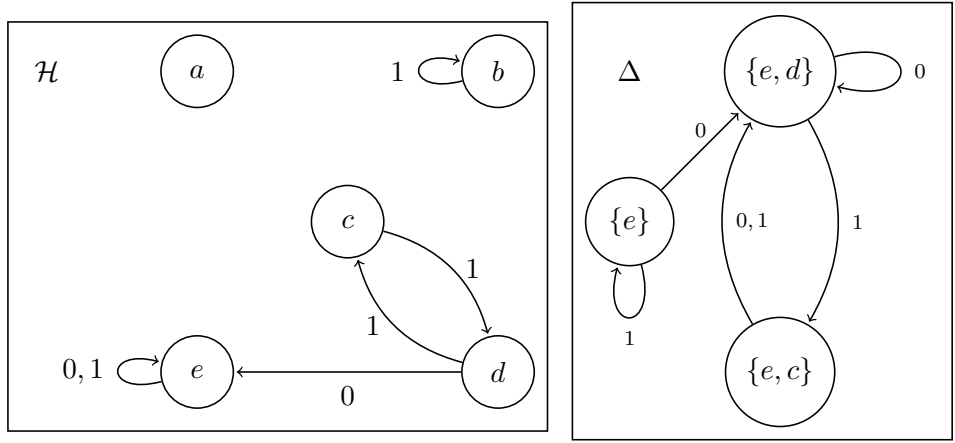
So $\mathcal C=\{e,b,c,d\}$ and $H_{11}=\{e,b,c,d\}$.  All vertices of $\Delta$ are minimal except $\{e\}$.  The element $e+b-c-d$ maps to $0$ under the projections $KH_{11}\to K[H_{11}/\langle c\rangle]$ and $KH_{11}\to K[H_{11}/\langle d\rangle]$ over any field, and so is a nontrivial element of the essential ideal.  Thus $K\mathcal G(G,\X)$ is not simple over any field and $\Oo(G,\X)$ is not simple.  This result was first proved by the first author (unpublished) and written up in~\cite{SS23}.

\newcommand{\etalchar}[1]{$^{#1}$}


\end{document}